\documentclass{article}
\usepackage{amsthm}
\usepackage{amsfonts}
\usepackage{amssymb}
\usepackage{amsmath}
\usepackage{nccmath}
\usepackage{mathtools}
\usepackage{mathrsfs}
\usepackage{setspace}
\usepackage{graphicx}
\usepackage{hyperref}
\usepackage[numbers,sort&compress]{natbib}
\usepackage{enumerate}
\usepackage{authblk}

\newtheorem{theorem}{Theorem}[section]

\theoremstyle{definition}
\newtheorem{definition}[theorem]{Definition}
\newtheorem{example}[theorem]{Example}

\theoremstyle{remark}
\newtheorem{remark}[theorem]{Remark}

\numberwithin{equation}{section}

\DeclareMathOperator{\Tr}{Tr}
\usepackage{amsmath,amssymb,graphicx} 
\usepackage[margin=1in]{geometry} 
\usepackage{enumerate}
\usepackage{authblk}
\usepackage{placeins}
\usepackage{array}
\usepackage{academicons}
\usepackage{xcolor}
\usepackage{svg}
\usepackage[utf8]{inputenc}
\usepackage{tikz}
\usepackage{subcaption}
\usepackage{fancyhdr}
\usepackage{pgfplots}
\usepackage{adjustbox}
\usepackage{multicol}
\usetikzlibrary{automata} 
\usetikzlibrary{positioning} 
\usetikzlibrary{arrows}
\usetikzlibrary{automata} 
\usetikzlibrary{positioning} 
\usetikzlibrary{arrows}

\DeclareMathOperator{\atanh}{arctanh}

\providecommand{\keywords}[1]{\textbf{\textit{Keywords:}} #1}
\providecommand{\subjclass}[1]{\textbf{\textit{MSC2020:}} #1}
\begin{document}

\nocite{*} 

\title{ New characterizations of the ring of the  split-complex numbers and  the field $ \mathbb{C} $ of complex numbers  and their  comparative analyses}

\author{Hailu Bikila Yadeta \\ email: \href{mailto:haybik@gmail.com}{haybik@gmail.com} }
  \affil{Salale University, College of Natural Sciences, \\
  Department of Mathematics, Fiche, Oromia, Ethiopia}
\date{\today}
\maketitle

\begin{abstract}
\noindent In this paper, we give  a new characterization of the split-complex numbers as a vector space $LC_2= \{xI+yE : x,y \in \mathbb{R},\, E^2=I \}$ of operators, where $I$ is the identity operator and $E$ is the unit shift operator that are operating on the space $\mathbb{P}_2$ of all real-valued $2$-periodic functions. We also characterize the field of the complex number $\mathbb{C}= \{x+yi: x, y \in \mathbb{R}, i^2 =-1 \}$ as the space of linear operators of the form $ \{xI+yE, E^2 = -I \}$,  where $I$ is the identity operator and $E$ the unit shift operator that are regarded as operating on the vector space $\mathbb{AP}_2 $ of all real-valued $2$-antiperiodic functions. In an analogy to the polar form of complex numbers, we form the hyperbolic form of some subset $ \mathcal{H} $ of the elements of $LC_2 $. We study some properties of the elements of $LC_2$, the trace, the determinant, invertibility conditions, and others. We study some elementary functions defined on subsets of $LC_2$ as compared and contrasted with the usual complex functions. We study properties like continuity, differentiability and define the holomorphic condition of $LC_2$ functions  in a different sense than complex functions. We establish the line integrals of the vector-valued functions in $LC_2$ and compare them against the well known results for complex functions of a complex variable.
\end{abstract}

\noindent\keywords{ $2$-periodic function, $2$-antiperiodic function, identity operator, shift operator, hyperbolic modulus, hyperbolic argument, spectrum, determinant, circulant matrix, Cauchy-Riemann type equations,  }\\
\\

\section{Introduction and preliminaries}

The split-complex numbers (or the hyperbolic numbers, motors, tessarines, perplex numbers, double numbers, and Lornenz numbers)  are based on the hyperbolic unit $j$ satisfying $j^2=1 $,  were introduced by James Cockle in 1848. See, for example, \cite{DV}, \cite{WP} and the articles cited therein. The set $ \{x+yj, x,y \in \mathbb{R} j^2=1\} $ of all split-complex numbers is denoted by $D$. See \cite{DV}.

 In this paper, we introduce a new characterization of the set of split-complex numbers as the set of operators  whose elements are of the form $L = xI+yE$ and that are linear combinations of the identity operator $I$ and the shift operator $E$ with their properties as operators operating on the space of all real valued periodic functions of fundamental period equal to $2$. We denote this vector space by $LC_2 $. We also show that  the space $LC_2$ has some analogy to the field of complex numbers $\mathbb{C}$. and study some similarities and differences between the spaces of complex valued function $f(z)$ of complex argument $z =x+yi $ where $i$ satisfying  $i^2=-1$, is the usual imaginary unit of the field of complex numbers, and operator valued function $f(L) $ with operator argument $L=xI+yE$, where $I$ is the identity operator and $E$  the shift operator defined on the space of all two periodic functions of fundamental period equal to $2$.

The motivation for the current work on new characterizations of the field of complex numbers as well as the ring of split-complex numbers is based on the author's previous study on the decomposition of periodic spaces into subspaces of periodic and subspaces of antiperiodic functions in \cite{HBY}.
In addition, we show that $LC_2$ is isomorphic to the ring of $2 \times 2 $ real circulant matrices. We denote by $LC^*_2$ the subset of $LC_2$ constituting those elements of $LC_2$ that are invertible operators. The subspace $LC_2^*$ form a group under multiplication, and that $LC_2^*$ is isomorphic to the subset $CM_2^*$ of all invertible $ 2 \times 2 $ circulant matrices, which forms a group under matrix multiplication. For some important discussion of circulant matrices, see \cite{KW}, \cite{HBY1}, \cite{HBY1} and the references therein. In this paper, we will present our main results in the next three sections as follows:
\begin{itemize}
  \item We  introduce  a new characterization of the ring $D$ of split-matrices as a ring $LC_2 $  of operators acting on the space of all periodic functions of period 2.    It is known that the set of all split-complex numbers is isomorphic to the set of all $2 \times 2 $ circulant matrices. However, we verify that the new characterization of the set of split-matrices as $LC_2$  bears same quantities like  traces and determinants of the circuant matrices. We also establish group isomorphism between some subspaces of $LC_2$ with corresponding parts of $CM_2$.
\item We introduce a new characterization  of the field of complex numbers as a vector space of operators operating on space of all antiperiodic functions of fundamental antiperiod 2. To the best knowledge of the author, this is a new characterization of the complex field.
  \item We establish some comparative analysis between the vector space $LC_2 $, which is a ring, and the field $\mathbb{C}$ of complex numbers. We compare some properties of $LC_2 $  function with the well-established results of complex functions.
\end{itemize}

\subsection{Shift operators and spaces of periodic functions}

\begin{definition}
For $ h \in \mathbb{R} $, we define the shift operator $ E ^h $ and the identity operator $ I $ as
$$ E ^h y(x) := y(x+h),\quad  I y(x) := y(x) .$$
 For $ h = 1 $, we write $ E^h $ only as $ E $ than $ E^{1}$. We agree  that $ E^0 = I $.
We define the forward difference operator $\Delta $ and the back ward difference operator $ \nabla $ as follows
$$ \Delta y(x):=  (E-I)y(x)= y(x+1) -y(x),\quad  \nabla y(x)= (I-E^{-1})y(x)= y(x)-y(x-1). $$
\end{definition}
The shift operators are very important in the study of difference equations and differential equations. The general solutions of linear difference equations are either periodic functions or linear combinations of some functions over periodic functions. See,for example, \cite{LB}, \cite{CHR},\cite{MT},\cite{KM}.
\begin{definition}
  A function $ f$ is said to be $p$-periodic if there exists a $ p >0 $ such that $ f(x)= f(x+ p),\,  x \in \mathbb{R}$. The least such $p$ is called the \emph{period} of $ f $. In terms of the shift operator, we write this as
  $$ E^pf(x) = f(x) .$$
We denote the space of all $p$-periodic functions by $ \mathbb{P}_p$.
\end{definition}

\begin{definition}\cite{GN}, \cite{JM}
  A function $ f$ is said to be $p$-antiperiodic if there exists a $ p >0 $ such that $ f(x+ p)= -f(x) ,\,  x \in \mathbb{R}$. The least such $p$ is called the \emph{ antiperiod} of $ f $. In terms of shift operators, we write this as
  $$ E^pf(x) = -f(x). $$
  We denote the space of all $p$-antiperiodic functions by $ \mathbb{AP}_p$.
  \end{definition}

\begin{example}
  The sequence of functions $ f_n(x)= \cos 2n \pi x , \, n\in \mathbb{N} $ are 1-periodic. The sequence of functions $g_n(x)= \cos (2n+1) \pi x ,\, n \in \mathbb{N} $ are 1-antiperiodic. The function $f(x)=x-\lfloor x \rfloor$, where $\lfloor x \rfloor$ denotes the greatest integer not greater than $x$, is a 1-periodic function.
\end{example}


\section{Anew characterization of the ring of split-complex numbers }
In this section, we characterize the ring of the split complex numbers as a set of some classes of operators acting on the space $LC_2 $ of all $2$-periodic functions.
\begin{definition}
Define a vector space $LC_2$ of operators over $\mathbb{R}$, operating on the space of all $2$-periodic functions, as
\begin{equation}\label{eq:LC2}
   LC_2: =\{ xI+yE: x,y \in \mathbb{R}, E^2=I \},
\end{equation}
where $I$ is the identity operator on $\mathbb{P}_2$, and $E$ is the unit shift operator on $\mathbb{P}_2 $. In (\ref{eq:LC2}), the real number $x$ is called the \emph{''identity component''} of $L \in LC_2 $, and the real number $y$ is called the \emph{''shift component''} of $L \in LC_2 $. The condition $E^2=I,$ follows from the fact that $E $ is supposed to operate on the space  $ \mathbb{P}_2 $ so that
$$ E^2f(x)=f(x+2)=f(x)= I f(x), \forall f \in \mathbb{P}_2 .$$
\end{definition}

 \begin{theorem}
  $LC_2$ is a vector space with vector addition defined as:
  \begin{equation}\label{eq:vectoraddition}
  (x_1I + y_1E) + (x_2I + y_2 E) =(x_1+x_2)I+(y_1+y_2)E,
  \end{equation}
  and scalar multiplication defined as:
\begin{equation}\label{eq:scalarmultiplication}
  c(xI+yE)= cxI+cyE,\quad c, x, y \in \mathbb{R}.
\end{equation}
\end{theorem}

\begin{proof}
   Let $ a, b \in \mathbb{R },\, L,M,N \in LC_2 $. We have the following properties:
  \begin{itemize}
    \item Closure property with respect to addition holds: $L+M \in LC_2 $,
    \item Closure property with respect to scalar multiplication holds:  $a L \in LC_2$,
    \item Commutative property of addition holds: $L+M= M + N $.
    \item Associative property with respect to vector addition holds: $(L+M)+N= L+(M+N) $,
    \item There is the zero operator $0I+0E:=\textbf{0} \in LC_2 $ such that $ \textbf{0}+L =L+\textbf{0}=L $,
    \item For each $ L\in LC_2 $, there exists an additive inverse $-L \in LC_2 $ such that $L + (-L) = \textbf{0 }$,
    \item  Associativity with respect to scalar multiplication holds, $a(bL)=(ab)L $,
    \item  Distributivity of a scalar over vector sum holds: $ a(L+M)= aL + aM $,
    \item  Distributivity of a vector over a scalar sum holds: $ (a+b) L = aL + bL $,
    \item  For the real number $1$, $1 L= L $.
  \end{itemize}
\end{proof}

\begin{definition}[\cite{NJ}]
  An \emph{algebra} is a linear space $X$ with a function called multiplication from $ X \times X $ to $X$, which has the following properties for all $x,y,z \in X, \alpha \in  \mathbb{R }$.
  \begin{itemize}
  \item Associativity $L(MN)=(LM)N$,
  \item Left distributivity, $L(M+N)=LM+LN$; and right distributivity $(M+N)L=ML+ NL$,
  \item  Homogeneity with respect to scalar multiplication: $(aL)(bM)= (ab L)M= L(abM)=ab(LM)$.\\
  If the multiplication is commutative, then the algebra is called \emph{commutative algebra}. An algebra is called an algebra with unity(identity) if there exists a nonzero element $I \in X $ such that
  $$ LI= IL=L. $$
\end{itemize}
  \end{definition}

\begin{definition}
For $L_1=x_1I+y_1E,\, L_2=x_2I+y_2E \in LC_2 $, an operation of multiplication can be defined naturally on $LC_2 $  as follows:
\begin{equation}\label{eq:operatormultiplication}
  L_1 L_2= (x_1I+y_1E)( x_2I+y_2 E) =(x_1x_2+ y_1 y_2)I+ (x_1y_2+ x_2y_1)E.
\end{equation}
\end{definition}

\begin{remark}
  The identity component, $ x_1x_2+y_1y_2 $, of the product $L_1L_2$ is exactly the same as the dot product of the vectors $(x_1,x_2)$ and $(x_2,y_2)$ in $ \mathbb{R}^2 $.
\end{remark}

\begin{theorem}
   $LC_2$ is a commutative algebra with unity.
\end{theorem}
\begin{proof}
  With $I$ as the identity, $E^2=I$ and $E^1=E $ commutativity, associativity, and distributive laws are simple to verify.
\end{proof}

\begin{theorem}
  $(LC_2,+, .) $ is a ring with the multiplication $.$ defined as in (\ref{eq:operatormultiplication}) and the addition $+$ defined as in (\ref{eq:vectoraddition}).
\end{theorem}

\begin{definition}
  An element $L=xI+yE \in LC_2 $ is said to be invertible if there exists an element $ M= aI+bE \in LC_2 $ such that
  \begin{equation}\label{eq:invertiblitycondition}
    (xI+yE)(aI+bE )=(ax+ by)I+(ay+ bx)E= I .
  \end{equation}
In this case, we say that $  M= aI+bE \in LC_2 $ is the inverse of $ L=xI+yE \in LC_2 $.
  \end{definition}

  \begin{theorem}\label{eq:invertiblitytheorem}
    The necessary and sufficient condition that an element $ xI+yE \in LC_2 $ is invertible is that $ x^2-y^2 \neq 0  $.
  \end{theorem}

  \begin{proof}
From (\ref{eq:invertiblitycondition}), we form a system of linear equations
\begin{equation}\label{eq:systemforinverse}
     ax+ by=1,\quad  ay+ bx=0.
  \end{equation}
  Solving (\ref{eq:systemforinverse}), we get $a= \frac{x}{x^2-y^2},\quad  b= \frac{-y}{x^2-y^2}$.
  Therefore,
  \begin{equation}\label{eq:Linverse}
     \frac{1}{L}= \frac{1}{xI+yE}= \frac{x}{x^2-y^2}I -\frac{y}{x^2-y^2}E,\quad x^2-y^2 \neq 0.
  \end{equation}
  Hence the theorem follows.
  \end{proof}

\begin{theorem}
   we have the following results:
\begin{itemize}
  \item $(EL)M=  E(LM)=E(ML)=(EM)L=L(EM)$
  \item (EL)(EM)=LM
\end{itemize}
\end{theorem}
\begin{proof}
  The result follows from commutativity and associativity in of multiplication in $LC_2$ and the fact that $E^2= I$.
\end{proof}

\subsection{Conjugate of the elements $L$  in $LC_2$ and their properties }
\begin{definition}
  The operator $ \overline{L}= xI-yE \in LC_2 $ is called the \emph{conjugate  operator} of the operator $ L =xI+ yE \in LC_2  $ and vice versa.
\end{definition}
\begin{theorem}
  Let $ a \in\mathbb{ R } $,  $ L= xI+yE $, and $M=rI+sE \in LC_2 $. We have the following properties of the conjugate:
  \begin{multicols}{2}
\begin{itemize}
  \item $\overline{aL}    =  a \bar{ L} .$
  \item $ \bar{ \bar{L}}= L .$
  \item $ \overline{L+M} = \bar{L}+ \bar{M}.$
  \item $ \overline{LM}= \bar{L}\bar{M} .$
  \item $ \overline{\left(\frac{L}{M} \right)}=  \frac{\bar{L} }{\bar{M}},\quad r^2-s^2 \neq 0.  $
\end{itemize}
\end{multicols}
\end{theorem}
Let $L=xI+yE \in LC_2 $. The identity component of $L$, and the shift component of $L$ may be calculated as follows:
\begin{equation}\label{eq:LandLbar}
   x = \frac{L+\bar{L}}{2I},\quad y=\frac{L- \bar{L}}{2E}.
\end{equation}

\subsection{Existence or nonexistence of fixed points of an element $L$}
\begin{definition}
Let  $L \in LC_2 $ is given. Consider the  map:
\begin{equation}\label{eq:factormap}
    L: LC_2 \rightarrow  LC_2, \quad M \mapsto LM .
\end{equation}

 An element $ M \in LC_2 $ is a fixed point of the map (\ref{eq:factormap}) if $LM=M $. Clearly, $\textbf{0}$ is a fixed point of the map (\ref{eq:factormap}) and is called the \emph{trivial fixed point}.
\end{definition}

\begin{theorem}
  Let $L \in LC_2 $. The set of all fixed points of $L$ form a group  with respect to addition.
  \end{theorem}

  \begin{proof}
    Let $\mathbb{S}$ be the set of all fixed points of $L$. $ \mathbb{S}$ is nonempty since $\textbf{0} \in \mathbb{S} $. Assume that $ M, N \in S $. Then
     \begin{itemize}
     \item  $L(M+N)=LM+LN=M+N $. So, $M+N \in \mathbb{S} $,
     \item $L(\alpha M)= \alpha LM= \alpha M, \forall \alpha \in \mathbb{R}  $. So, $ \alpha M \in \mathbb{S} $.
     \end{itemize}
   \end{proof}

\begin{theorem}
The necessary and sufficient condition that an element $L=xI+yE$ of $LC_2 $ has a nontrivial fixed point is that $x \pm y=1$.
\end{theorem}

\begin{proof}
  Let $L=xI+yE $. If $aI+bE$ is a nontrivial  fixed point of $L$, then
  $$ (xI+yE)(aI+bE)= aI+bE $$ from which, by equating the $I$  and the $E$ components, we have the a homogeneous system
  \begin{equation}\label{eq:homogeneousystem}
      \begin{bmatrix}
        x-1 & y \\

        y & x-1
      \end{bmatrix} \begin{bmatrix}
                      a \\
                      b
                    \end{bmatrix} =\begin{bmatrix}
                      0 \\
                      0
                    \end{bmatrix}
      \end{equation}
 The homogenous system (\ref{eq:homogeneousystem})  has a non trivial solution if $x \pm y=1 $.
\end{proof}

Let
\begin{equation}\label{eq:setS1}
  \mathbb{ FS}_1= \{ xI+yE \in LC_2, x-y=1 \},
\end{equation}

\begin{equation}\label{eq:setS2}
   \mathbb{FS}_2= \{ xI+yE \in LC_2, x+y=1 \},
\end{equation}

\begin{equation}\label{eq:spaceS1}
  \mathbb{ S}_1= \{ xI+yE \in LC_2, x+y=0 \},
\end{equation}

\begin{equation}\label{eq:spaceS2}
   \mathbb{S}_2= \{ xI+yE \in LC_2, x-y=0 \},
   \end{equation}

\begin{theorem}
  $\mathbb{S}_1$ and $\mathbb{S}_2$ are subgroups of $LC_2$ under the operation of  addition.
\end{theorem}

\begin{theorem}
 We have the following results:
 \begin{itemize}
   \item For each $L\in \mathbb{FS}_1$, for any $ M \in \mathbb{S}_1 $  we have $LM=M$.
   \item For each $L\in \mathbb{FS}_2$, for any $ M \in \mathbb{S}_2 $  we have $LM=M$.
   \item $LC_2 = \mathbb{S}_1 \oplus \mathbb{S}_2 $, and $  \mathbb{FS}_1 \cap \mathbb{FS}_2= \{ I \}$ fixes every element of $LC_2$.
 \end{itemize}
\end{theorem}

Unlike the operation of addition, $\mathbb{S}_1$ and $\mathbb{S}_2$ will not form subgroups under the operation of multiplication. However, we have the following result:

\begin{theorem}
  Each of $\mathbb{S}_1$ and $\mathbb{S}_2$ is a subring of $LC_2$ which is also an ideal of $LC_2$ .
\end{theorem}
\begin{proof}
   Let $L= xI+yE,\, L'= x'I+y'E \in  \mathbb{S}_1 $. Then $x+y=0, \, x'+y'=0 $.For the product $LL'=(xx'+yy')I+(xy'+yx')E $ we have $ (xx'+yy')+(xy'+yx') = (x+y)(x'+y')=0 $, verifying that $ LL' \in \mathbb{S}_1 $. The proof is similar for $ \mathbb{S}_2 $
\end{proof}

\begin{theorem}
  $\mathbb{S}_1$ and $\mathbb{S}_2$ are closed under the operation of multiplication and uniarticular,
  $L^n =  (2x)^{n-1}L$ whenever $L= xI+xE \in \mathbb{S}_1 $ or $L= xI-xE \in \mathbb{S}_2 $.
\end{theorem}

\subsection{ The element $L$ that are zero divisors }
\begin{definition}
A ring $R$ is said to have a \emph{zero divisor} if there exist two non-zero elements $x$ and $y$ of $R$ such that $xy=0$. In this case, both $x$ and $y$ are said to be zero divisors. A non-trivial ring with no zero divisor is called an \emph{integral domain.}
  \end{definition}
Clearly, $LC_2 $ has zero divisors, and hence it is not an integral domain. For example,
$$(E-I)(E+I)=E^2-I=0. $$
\begin{theorem}
  Let $L=aI+bE$ and $ M=cI+ d E $ be non-zero, and that $LM=0 $.  The sufficient condition that $L$ and $M$ are zero dividers  is that one of the points $(a,b)$ and $(c,d)$ is on the line $y=x$ and the other is on the line $y =-x $.
\end{theorem}
\begin{proof}
  $LM=0$ implies $ac+bd=0$ and $ ad+bc=0$. From the first equation, we get $\frac{c}{d}= -\frac{b}{a}=\lambda $. plugging  this into the second equation, we get
  $$ ad(\lambda^2 -1)=0. $$
  $\lambda =1  $ implies that $(a,b)$ is on the line $y = x$ and  $(c,d)$ is on the line $y= -x$. $\lambda = -1  $ implies that $(a,b)$ is on the line $y = -x$ and  $(c,d)$ is on the line $y= x$. Other nonzero elements of $LC_2$ are of the form $xI$ or $yE$ with $xy \neq 0 $. Except for the identity $I$ which fixes all elements, these have no nontrivial fixed points. Therefore, all the maps of the form that have nontrivial fixed points are those with $L=xI+yE, x^2-y^2=0, x^2 +y^2 \neq 0 $ and $L=I$.
\end{proof}

\begin{theorem}
  $$  L \in \mathbb{FS}_i \Leftrightarrow L-I \in S_{3-i},\, i=1,2. $$
\end{theorem}
\begin{proof}
  \begin{align*}
    L \in \mathbb{FS}_i & \Leftrightarrow LM=M,\,  \forall M \in \mathbb{S}_i  \\
     & \Leftrightarrow (L-I)M=0,\, \forall M \in \mathbb{S}_i \\
     & \Leftrightarrow (L-I) \in \mathbb{S}_{3-i},\, i=1,2.
  \end{align*}
\end{proof}

\subsection{ Elementary functions of $L$ }
 Following the usual Taylor's series expansion for exponential functions and the trigonometric functions sines, and cosines, we define the exponentials, the sines and the cosines of $L$  in this subsection.
\subsubsection{Exponential function}
\begin{definition}
  For $L \in LC_2 $ we have
  \begin{equation}\label{eq:exponentialofL}
    e^L:= \sum_{n=0}^{\infty}\frac{L^n}{n!}
  \end{equation}
\end{definition}

\begin{theorem}
For $L =xI+yE \in LC_2 $,  we have
\begin{equation} \label{eq:exponentialform}
  e^{L}= e^ {x} \cosh (y) I +  e^ {x} \sinh (y) E.
\end{equation}
\end{theorem}

 \begin{proof}
 \begin{align*}
  e^{xI }&= \sum_{n=0}^{\infty}\frac{x^nI^n}{n!}= I \sum_{n=0}^{\infty}\frac{x^n}{n!}=  e^{x}I , \\
   e^{yE }&= \sum_{n=0}^{\infty}\frac{y^nE^n}{n!}=  \sum_{n=0}^{\infty}\frac{y^{2n} E^{2n}}{(2n)!} + \sum_{n=0}^{\infty}\frac{y^{2n+1}E ^{2n+1} }{(2n+1)!} =  \cosh (y) I+ \sinh (y) E.
  \end{align*}
 Therefore, for $L =xI+yE \in LC_2 $  we have
$$ e^{L}= e^ {xI+yE}= e^ {xI}e^ {yE} = e^ {x} \cosh (y) I +  e^ {x} \sinh (y) E. $$
\end{proof}

\subsubsection{Trigonometric functions}
\begin{definition}
  For $L \in LC_2 $ we have
  \begin{equation}\label{eq:sineandcosine}
  \cos L:= \sum_{n=0}^{\infty} (-1)^n \frac{L^{2n}}{(2n)!},\quad  \sin L:= \sum_{n=0}^{\infty} (-1)^n \frac{L^{2n+1}}{(2n+1)!}
\end{equation}
\end{definition}

\begin{theorem}
  Let $ x, y \in \mathbb{R} $. We have the following identities
  \begin{equation}\label{eq:cosxI}
   \cos(xI)= \cos (xE)= \cos (x)I,
  \end{equation}
  \begin{equation}\label{eq:sinxIandsinxE}
   \sin (xI)= \sin(x)I, \quad  \sin(xE)= \cos(x) E.
  \end{equation}
\end{theorem}

\begin{proof}
We use the Maclaurin's series expansions of cosine and sine  defined in (\ref{eq:sineandcosine}) and the fact that $E^2=I$, we get
$$ \cos(xE)= \sum_{n=0}^{\infty}   \frac{(-1)^n}{(2n)!} x^{2n} E^{2n} = \sum_{n=0}^{\infty}   \frac{(-1)^n}{(2n)!} x^{2n} I =   \cos(x) I,   $$

$$ \sin(xI)= \sum_{n=0}^{\infty} \frac{(-1)^n}{(2n+1)!} x^{2n+1} I^{2n+1}= \sum_{n=0}^{\infty} \frac{(-1)^n}{(2n+1)!} x^{2n+1} I = \sin (x) I,   $$

$$ \sin(xE)= \sum_{n=0}^{\infty} \frac{(-1)^n}{(2n+1)!} x^{2n+1} E^{2n+1}= \sum_{n=0}^{\infty} \frac{(-1)^n}{(2n+1)!} x^{2n+1} E = \sin (x) E . $$
\end{proof}

\begin{theorem}
   \begin{equation}\label{eq:sinL}
   \sin(xI+yE)=  \sin(x)\cos(y) I+\cos(x)\sin(y) E
  \end{equation}
  \begin{equation}\label{eq:cosL}
  \cos(xI+yE)= \cos(x)\cos(y) I- \sin(x)\sin(y) E
  \end{equation}
\end{theorem}

\begin{proof}
Assuming that the usual trigonometric identities of sine and cosine are also applicable, we have
\begin{itemize}
   \item $ \sin (xI+yE) = \sin(xI)\cos(yE)+\cos(xI)\sin(yE) = \sin(x)\cos(y) I+\cos(x)\sin(y)E   $

\item $ \cos (xI+yE)= \cos(xI)\cos(yE)-\sin(xI)\sin(yE) = \cos(x)\cos(y) I- \sin(x)\sin(y)E  $
  \end{itemize}
\end{proof}

\begin{example}
Let us find the points $ L \in LC_2 $  where $\sin L $ and $ \cos L $ are not invertible. According to Theorem \ref{eq:invertiblitytheorem}, the points $ L \in LC_2 $  where  $ \sin L $ is not invertible are
$$ \{ xI+yE: x+y = n \pi,    n \in \mathbb{Z }\}  \cup  \{xI+yE: x-y = n \pi,   n \in \mathbb{Z } \}.  $$
The set of points $ L \in LC_2 $  where  $ \cos L $   is not invertible is
$$ \{  xI+yE: x+y = (2n+1) \pi/2,   n \in \mathbb{Z } \}  \cup  \{ L = xI+yE: x-y = (2n+1) \pi/2,  n \in \mathbb{Z }\}  .$$
\end{example}
\subsubsection{Square root of $L= xI+yE \in LC_2 $}
\begin{definition}
   $L= aI+bE \in LC_2 $  is a square root of  $L= xI+yE \in LC_2 $ if
   \begin{equation}\label{eq:Lsquared}
     L^2 = ( aI+bE) ^2 = (a^2+b^2)I+ 2ab E = xI+yE.
   \end{equation}
In this case we write as $ \sqrt{xI + yE}= aI + bE. $
\end{definition}
From (\ref{eq:Lsquared}) we can calculate that
\begin{equation}\label{eq:squarerootequation}
   \sqrt{xI + yE}= \frac{1}{2} ( \sqrt{x+y}+ \sqrt{x-y})I + \frac{1}{2} ( \sqrt{x+y}- \sqrt{x-y})E,
\end{equation}

provided that $x+y \geq 0, \, x-y \geq 0 $. That is, $ xI +yE \in \mathcal{E}$.

\subsubsection{Logarithmic function}
\begin{theorem}
  Let $L= xI+ yE \in LC_2,\, x^2-y^2 > 0,\, x > 0 $. We have the following identities
  \begin{equation}\label{eq:logartiminhyperbolic}
   \ln (L)=\ln (xI+yE)=  \frac{1}{2}\ln(x^2-y^2)I + \atanh (y/x)E
  \end{equation}
\end{theorem}
\begin{proof}
  Let $\ln( xI+yE)= \alpha I + \beta E $. Then  $ xI+yE= e^{\alpha I + \beta E } $.
  From (\ref{eq:exponentialform}) it follows that $xI+yE= e^{\alpha} ( \cosh \beta I + \sinh \beta E) $. So, equating the coefficients of $I$ and $E$ we get
  \begin{equation}\label{eq:lnparametrized}
     x= e^{\alpha} \cosh \beta, \quad y=  e^{\alpha} \sinh \beta.
  \end{equation}
The desired result (\ref{eq:logartiminhyperbolic}) follows from  (\ref{eq:lnparametrized}).
we have $x^2-y^2= e^{2\alpha} > 0 $ and also $x = e^{\alpha} \cosh \beta > 0 $. Therefore the domain of the logarithm $ \ln L $ is
$$  \{xI+yE \in LC_2 : x > |y| \}  $$
\end{proof}

\subsection{ $ \mathbb{C} $ is a field,  $LC_2$ is a ring that is  not  a field }
 We have seen that $(LC_2,+) $ is an Abelian group, that $ (LC_2,+,.) $ is a ring, which is a vector space over $ \mathbb{R} $, and that is an associative algebra. It satisfies many of the properties of the field, that is a commutative ring all of whose non zero elements have multiplicative inverse. The problem is that for the case of $LC_2 $, not only the zero element zero  is a not invertible. Rather, the set of points $xI+yE \in LC_2 $ with  the condition $x^2-y^2=0 $  consists of infinitely many points that are not invertible. This means that $LC_2 $ has nonzero elements that are not invertible.
$LC_2$ isomorphic to the ring $ CM_2 $ of  all $ 2 \times 2 $ real  circulant matrices  which is not a field. See \cite{WP} and the references cited therein.

\subsection{ Spectrum, trace, and determinant of an operator $L=aI+bE \in LC_2 $}
We calculate the traces  and determinant of a split-complex number in its new characterization as a linear operator acting on spaces of 2-antiperiodic functions and verify that they are exactly the same as the trace and determinant of a cirulant matrix.

\begin{theorem}
Let $L=aI+bE \in LC_2 $.  Then the eigenvalues $L $  are the numbers $ \lambda_1= a+b $ and $ \lambda_2= a-b $.
  \end{theorem}

\begin{proof}
Let $\lambda $ be an eigenvalue of an operator $L=aI+bE \in LC_2 $. Then we have
 \begin{equation}\label{eq:eigenequation1}
    (aI+bE)f= \lambda f,
 \end{equation}
for some $ f\in \mathbb{P}_2, f \neq 0 $. Applying the shift operator $ E $, we get
 \begin{equation}\label{eq:eigenequation2}
    ( aE+bI)f= \lambda Ef.
 \end{equation}
 From (\ref{eq:eigenequation1}) and (\ref{eq:eigenequation2}), we get the system of equations
  \begin{equation}\label{eq:matrixsystem}
    \begin{pmatrix}
   a-\lambda & b   \\
   b &  a-\lambda
 \end{pmatrix}  \begin{pmatrix}
                  f \\
                  Ef
                \end{pmatrix}=\begin{pmatrix}
                                0 \\
                                0
                              \end{pmatrix} .
  \end{equation}
This homogeneous system has a nontrivial solution, which is an eigenvector of the operator $L=aI+bE $, only if the determinant of the coefficient matrix is zero. That is when
$$ \lambda= \lambda_1 = a+b, \quad \lambda = \lambda_2= a-b. $$
\end{proof}

\begin{theorem}
  The eigenspace of $L=aI+bE $ corresponding to the eigenvalue $\lambda = a+b $ is the set $\mathbb{P}_1 $  of all 1- periodic functions. The eigenspace  of $L=aI+bE $ corresponding to the eigenvalue $\lambda = a-b $  is the set $\mathbb{AP}_1 $   of all 1-antiperiodic functions.
  \end{theorem}

\begin{proof}
If $f \in \mathbb{P}_1 \subset \mathbb{P}_2 $, then $ (aI+bE)f= (a+b)f$, and if $f \in \mathbb{AP}_1 \subset \mathbb{P}_2 $ then  $ (aI+bE)f= (a-b)f $.  In \cite{HBY}, it was shown that
\begin{equation}\label{eq:eigendecomposition}
   \mathbb{P}_2 =  \mathbb{AP}_1 \oplus\mathbb{ P}_1.
\end{equation}
Therefore, $ \mathbb{P}_1 $ is the eigenspace corresponding to eigenvalue $a+b$, and  $ \mathbb{AP}_1 $ is the eigenspace corresponding to eigenvalue $a-b$.
\end{proof}

\begin{remark}
  The direct sum decomposition (\ref{eq:eigendecomposition}) is in fact an eigen-decomposition.
\end{remark}

\begin{definition}
  Let $L =aI+aE \in LC_2 $.  The\emph{ spectrum }of the operator $L$, denoted by $\sigma(L) $, is the set of all eigenvalues of $L$. The \emph{determinant} of $L$, denoted by $\det (L) $, is the product of all the eigenvalues of $L$. The \emph{trace} of $L$, denoted by $\Tr(L) $, is the sum of all the eigenvalues of $L$.
 \end{definition}

\begin{theorem}
Let $aI+bE \in LC_2 $ be an operator on $\mathbb{P}_2 $. Then we have the following results:
  \begin{equation}\label{eq:spectrum}
   \sigma(L)=\{ a-b,\, a+b\}.
 \end{equation}
\begin{equation}\label{eq:trace}
   \Tr (L)= \sum_{\lambda \in \sigma(L)}\lambda = 2a
\end{equation}

\begin{equation}\label{eq:determinant}
   \det (L) = \prod_{ \lambda \in \sigma(L)} \lambda =a^2-b^2
\end{equation}
 \end{theorem}
\begin{remark}
Note that both the sum formula (\ref{eq:trace}) and the product formula (\ref{eq:determinant}) take into account the multiplicities of repeated eigenvalues.
\end{remark}

Throughout this paper, we use the following notations for the indicated subsets of $LC_2$
\begin{equation}\label{eq:LC2star}
 LC^*_2 :=\{L= xI+yE \in LC_2: x^2-y^2 \neq 0  \} ,
\end{equation}
\begin{equation}\label{eq:H}
 \mathcal{H}:=\{L= xI+yE \in LC_2: x^2-y^2 > 0  \},
\end{equation}
\begin{equation}\label{eq:V}
  \mathcal{V}:=\{L= xI+yE \in LC_2: x^2-y^2 < 0  \},
\end{equation}
\begin{equation}\label{eq:U}
  \mathcal{U}:=\{L= xI+yE \in LC_2: x^2-y^2 = 1  \},
\end{equation}
\begin{equation}\label{eq:N}
  \mathcal{N}:=\{L= xI+yE \in LC_2: x^2-y^2 = 0  \}.
\end{equation}
\begin{equation}\label{eq:E}
  \mathcal{E}: = \{ xI+yE \in \mathcal{H }: x > 0 \}.
\end{equation}

 \begin{theorem}\label{eq:subgrouptheorem}
Then we have the following results:
\begin{itemize}
 \item    Each of $LC^*_2 $,  $\mathcal{H}$, $ \mathcal{U} $, $ \mathcal{E} $  is a group under multiplication, and  $ \mathcal{U} \lhd \mathcal{H} \lhd LC^*_2 $,

  \item $ \mathcal{V} =  \{E L: L \in \mathcal{H} \},\quad  \mathcal{H} =  \{E M: M \in \mathcal{V} \}$,

  \item  For any $ L \in LC_2 $, for any $ N \in \mathcal{N} $, $ LN = NL \in \mathcal{N}$.
  \end{itemize}
\end{theorem}

\begin{proof}
The proofs follow from the multiplicative property of determinants. If $ L_1= x_1I+y_1E, L_2 = x_2I+y_2E $, then
\begin{align*}
  \det ( L_1L_2)&= \det( (x_1x_2+y_1y_2)I+(x_1y_2+x_2y_1)E )\\
   &= (x_1x_2+y_1y_2)^2 -(x_1y_2+x_2y_1)^2 \\
   & = (x_1^2x_2^2+2x_1y_1x_2y_2+y_1^2y_2^2)-(x_1^2y_2^2+2x_1y_1x_2y_2+x_2^2y_1^2) \\
   & =x_1^2(x_2^2-y_2^2)+y_1^2(y_2^2-x_2^2)\\
   &=(x_1^2-y_1^2)(x_2^2-y_2^2)\\
  & =\det L_1 \det L_2
\end{align*}

\begin{itemize}
\item The identity $I$ belongs to all $ LC^*_2, \, \mathcal{H}, \, \mathcal{U}$, and $\mathcal{E}$.
\item If $ L_1, L_2 \in \mathcal{H}$, then $ \det L_1 >0, \det L_2 >0 $. Consequently, $ \det (L_1L_2)= \det L_1 \det L_2 >0 $. This shows that $ L_1L_2 \in  \mathcal{H} $. Also $\det (L_1^{-1}) =1 / \det (L_1) >0 $. Therefore, $L_1^ {-1} \in \mathcal{H} $.

 \item If $ L_1, L_2 \in LC^*_2 $, then $ \det L_1 \neq 0 , \det L_2 \neq 0 $. Consequently, $ \det (L_1L_2)= \det L_1 \det L_2 \neq 0 $. This shows that $ L_1L_2 \in  \in LC^*_2 $. Also $\det (L_1^{-1}) =1 / \det (L_1) \neq 0   $. Therefore, $L_1^ {-1} \in LC^*_2 $.

 \item If $ L_1, L_2 \in \mathcal{U}$, then $ \det L_1 =1  , \det L_2 =1 $. Consequently, $ \det (L_1L_2)= \det L_1 \det L_2 =1 $. This shows that $ L_1L_2 \in  \mathcal{U} $. Also $\det (L_1^{-1}) =1 / \det (L_1)$. Therefore, $L_1^ {-1} \in \mathcal{U}$.

\item If  $ L = x I+y \in \mathcal{E} $, then  by (\ref{eq:H}) and (\ref{eq:E}),  $ x^2-y^2 > 0, x> 0 $. This is true if and only if $ x > |y| $.  Now for any $L_1= x_1I + y_1 E, L_2= x_2 I + y_2 E \in \mathcal{E} $, since $x_1> |y_1|  $, and $x_2 > |y_2| $, we have $ x_1x_2 > |y_1||y_2| \geq -y_1y_2 $. So, $x_1x_2+y_1y_2 > 0 $. This implies that $L_1L_2=(x_1x_2+y_1y_2)I + (x_1y_2+x_2y_1)E \in  \mathcal{E} $.
\end{itemize}
\end{proof}

\begin{theorem}
  \begin{equation}\label{eq:etotheL}
   \mathcal{E}= \{ e^L: L \in LC_2\}
  \end{equation}
\end{theorem}
\begin{proof}
If $L=xI+yE \in LC_2 $ then $ e^L= e^x\cosh(y) I + e^x \sinh(y) E \in \mathcal{E }  $. Conversely, if $ M \in \mathcal{E} $ then  by (\ref{eq:logartiminhyperbolic}) $\ln (M) \in LC_2 $ and  $ M = e ^{\ln (M)} \in \{e^{L}: L \in LC_2 \} $.
\end{proof}

\subsection{Isomorphism between the ring of operators $LC_2$ and the ring $CM_2$ of circulant matrices }
\begin{definition}
  The set $ CM_2 $  of all $ 2 \times 2 $ real circulant matrices is defined as
\begin{equation}\label{eq:2by2circulant}
     CM_2= \left\{ C=\begin{pmatrix}
   a & b   \\
   b & a
 \end{pmatrix},\, a,b \in \mathbb{R} \right\}.
   \end{equation}
  \end{definition}
 The spectrum, the trace, and the determinant of the operator $L=aI+bE  \in LC_2 $ are respectively the same as the spectrum, the trace, and the determinant of the corresponding circulant matrix $ \begin{pmatrix}
   a & b   \\
   b & a
 \end{pmatrix} \in CM_2 $. Also, the condition of invertibility of $L$  stated in Theorem \ref{eq:invertiblitytheorem}  is the same as the condition for the invertibility of a $ 2 \times 2 $ circulant matrix $ C \in CM_2$ given in (\ref{eq:2by2circulant}). That is the determinant $ \det (C)= a^2-b^2 \neq 0 $
Let us denote
\begin{equation}\label{eq:cm2star}
  CM_2^* := \{C \in CM_2,\quad \det(C)\neq 0  \},
\end{equation}

\begin{equation}\label{eq:cm2plus}
  CM_2^+ := \{C \in CM_2,\quad \det(C)>0  \},
\end{equation}

\begin{equation}\label{eq:cm2plusplus}
  CM_2^{++}:= \{C \in CM_2^{+},\quad  a >0 \}.
\end{equation}

\begin{equation}\label{eq:cm2one}
 CM_2^1 := \{C \in CM_2, \quad\det(C)=1 \} ,
\end{equation}

\begin{equation}\label{eq:cm2zero}
CM_2^0:= \{C \in CM_2,\quad \det(C)=0  \},
\end{equation}

\begin{equation}\label{eq:cm2minus}
  CM_2^-  := \{C \in CM_2,\quad \det(C)<0  \},
\end{equation}

\begin{theorem}
  $CM_2 $ is a ring under  the usual matrix multiplication and matrix addition, and that we have the following ring isomorphism:
  \begin{equation}\label{eq:ringisomorphism}
    (CM_2, +,.) \cong (LC_2,+,.)
  \end{equation}

\end{theorem}
    \begin{theorem}
$CM_2^*$ form a group under matrix multiplication, and each of $  CM_2^+ $,  $ CM_2^{++}$ , $CM_2^1$ form a subgroup of $CM_2^*$ under matrix multiplication.
    \end{theorem}

 \begin{theorem}
  We have the following group isomorphism with respect to multiplication:
    \begin{equation}\label{eq:isomorphism}
     ( CM_2^*,.)  \cong (LC_2 ^*,.), \quad (CM_2^1,.)  \cong (\mathcal{U},.),\quad  (CM_2^+,.)  \cong (\mathcal{H},.), (\quad CM_2^{++}, .) \cong  (\mathcal{E},.).
    \end{equation}
\end{theorem}

\subsection{Norms can be defined on $LC_2$}
 We have seen that $LC_2 $ is a vector space over $\mathbb{R}$, with the given vector addition (\ref{eq:vectoraddition}) and scalar multiplication (\ref{eq:scalarmultiplication}). We may define a norm on $ LC_2 $ to make it a normed vector space. For example, we may define the norms:
 \begin{equation}\label{eq:L1norm}
  \|.\|_1:  LC_2 \rightarrow [0,\infty),\quad \|xI+yE\|_1:=|x|+|y|,
 \end{equation}
 \begin{equation}\label{eq:L2norm}
 \|.\|_2: LC_2 \rightarrow [0,\infty),\quad  \|xI+yE\|_2:=\sqrt{x^2+y^2},
 \end{equation}
 \begin{equation}\label{eq:L2norm}
 \|.\|_\infty: LC_2 \rightarrow [0,\infty),\quad  \|xI+yE\|_\infty:=\max   \{|x|,  |y| \}.
 \end{equation}
Defining a norm on $LC_2$ is helpful to study some more properties of  $LC_2$ functions, like continuity of functions, convergence of sequences,etc.. For example, for $L= xI+yE $,
 \begin{equation}\label{eq:finitegeometric}
     I +L+L^2 +L^3+...+L^n  = \frac{I-L^{n+1}}{I-L}, \quad \|L\|_1 < 1,
 \end{equation}
and the infinite geometric series
 \begin{equation}\label{eq:infinitegeometric}
    I +L+L^2 +L^3+...  = \frac{I}{I-L}, \quad \|L\|_1 < 1 .
 \end{equation}
  Equations(\ref{eq:finitegeometric}) and (\ref{eq:infinitegeometric}) are meaningful as $I-L $ is invertible and its inverse is in $LC_2$ when $\|L\|_1 < 1  $.
 \begin{remark}
 The  results for the finite geometric series (\ref{eq:finitegeometric})  and the infinite geometric series (\ref{eq:infinitegeometric}) are not valid for the norms $\|.\|_2$ and $ \|.\|_\infty $.
 \end{remark}
In the subsequent discussions, we use the norm $\|.\|_2$ unless specified, and just denote as $|.|$ for simplicity.

\section{A new characterization of the complex  field $\mathbb{C}$ }
In this section, we introduce a new characterization of the field of complex numbers as the set of some class of operators defined on spaces of all antiperiodic functions of antiperiod 2. We first list some of the commonest characterizations of the field of complex numbers commonly used.
\subsection{Rectangular coordinate form of complex numbers}
\begin{definition}
The field of complex numbers is  defined naturally as
\begin{equation}\label{eq:complexusual}
    \mathbb{C} = \{ x+iy: x,y \in \mathbb{R },\quad i^2= -1 \}.
\end{equation}
This is the rectangular coordinate form of the complex numbers.
\end{definition}

\subsection{Polar coordinate form of complex numbers}
In the polar form of the complex number we write
\begin{equation}\label{eq:polarform}
 z= x +iy = \rho e^{i\theta},\, \rho= \sqrt{x^2+y^2},\, \theta = \arg(z)
\end{equation}

\begin{definition}
The \emph{principal argument} of a complex number $z \neq 0 $, denoted by $\arg z $, is a unique angle $\theta $, where  $ - \pi \leq \theta < \pi $  is measured from the positive $x$-axis to the vector $\vec{oz}$.
\end{definition}

\begin{equation}\label{eq:principalargument}
   \arg(z)=
\begin{cases}
  \arctan(y/x) , & \mbox{ if } x >0 \\
  \arctan(y/x)+ \pi, & \mbox{if } x< 0, y > 0, \\
 \arctan(y/x)- \pi & \mbox{if } x < 0, y \leq 0 \\
  \pi /2  & \mbox{if } x=0,y>0 \\
   -\pi /2 & \mbox{ if } x=0,y <0, \\
   \text{undefined} & \mbox{ if } x=0,y=0.
\end{cases}
\end{equation}

\subsection{The complex field is isomorphic to  a field  of $2\times 2 $ matrices}

It is shown that $LC_2$ is isomorphic to the ring $CM_2 $ of all $2 \times 2 $ circulant matrices. $CM_2$ is a ring and  not a field, as there are infinitely many nonzero non-invertible elements, unlike a field where the only non invertible element is the zero element. Consider the space of all $2 \times 2 $  matrices of the form
\begin{equation}\label{eq:2by2antisymmetric}
    \mathbb{AS}_2:= \left\{ \begin{pmatrix}
   x & -y   \\
   y & x
 \end{pmatrix},\, x,y \in \mathbb{R} \right\} .
   \end{equation}
 The matrix field   $\mathbb{AS}_2$ is a isomorphic to the field of complex numbers $\mathbb{C}$. See, for example, \cite{Et}, \cite{RR}.

\subsection{A new characterization of the complex field $\mathbb{C}$ as set of operators on some space}

\begin{definition}
In an analogous manner to the definition of $LC_2 $ defined in (\ref{eq:LC2}), we may define the field $\mathbb{C}$ of complex numbers  as:
  \begin{equation}\label{eq:complexoperatorform}
    \mathbb{C} = \{ xI+ yE: x,y\in \mathbb{R }, E^2= -I \},
\end{equation}
where $I$ is the identity operator, and $E$ is the shift operator, both are considered to be acting on the space $ \mathbb{AP}_2 $ of all antiperiodic functions of antiperiod equal to $2$,  whereas for  the case of $LC_2$, the operators $E$ and $I$ are assumed to operate on  the space $\mathbb{P}_2$ of all periodic functions of period equal to $2$.
\end{definition}

\begin{theorem}
  The set complex numbers  vector space of operators (\ref{eq:complexoperatorform}) is a field.
\end{theorem}

\section{Comparison of the  field $\mathbb{C}$ of complex numbers  and $ LC_2 $}

\subsection{The polar form of $\mathbb{C}$ verses the hyperbolic form of $LC_2$}
In this paper, we attempt to write the elements of the subset $\mathcal{H} $ of $ LC_2 $ in exponential form. However, it requires some adjustment to write every element of $LC_2 $ in exponential form in an analogous manner to  writing any element $z$ of the field of complex numbers in polar form.

\begin{theorem}
  Let  $ L= xI+yE \in \mathcal{H} \subset  LC_2 $. Then $L$ can be written in the form
  \begin{equation}\label{eq:hyperbolicform}
    xI+yE= \breve{\rho}e^{\theta E},\quad \breve{\rho}=\sqrt{x^2-y^2},\quad \theta = \atanh (y/x)
  \end{equation}
\end{theorem}

\begin{proof}
\begin{equation}\label{eq:hyperbolicmodulus}
 xI+yE= \sqrt{x^2-y^2}\left( \frac{x}{\sqrt{x^2-y^2}} I+ \frac{y}{\sqrt{x^2-y^2}}E\right)
\end{equation}
Now define
  \begin{equation}\label{eq:hyperboliccosineandsine}
    \frac{x}{\sqrt{x^2-y^2}}:= \cosh(\theta),\quad \frac{y}{\sqrt{x^2-y^2}}:= \sinh(\theta)
  \end{equation}
Equations in (\ref{eq:hyperboliccosineandsine}) satisfy $ \cosh ^2(\theta)-\sin^2(\theta)=1$, and $\tanh (\theta) = \frac{y}{x}$. From which we can calculate that
\begin{equation}\label{eq:theta}
   \theta = \atanh (y/x) = \frac{1}{2}\ln\left(\frac{1+y/x}{1-y/x}\right)= \frac{1}{2}\ln \left(\frac{x+y}{x-y}\right)
\end{equation}
For subspace $\mathcal{H }$ on which  we defined hyperbolic coordinate form, we have $|y/x|<1$, which coincide with the natural domain of $\atanh $. However, in  $ \mathcal{V}\subset LC_2 $  $\tanh^{-1} $ attains complex values, and in $\mathcal{N}$  $\tanh^{-1} $  is undefined.
\end{proof}

We have the following facts:
\begin{itemize}
  \item $\theta = -\infty $ corresponds to the line $y+x =0 $,
  \item $\theta = \infty $ corresponds to the line $y-x =0 $,
  \item  $\theta = 0 $ corresponds  to the line $y =0 $ or the $x$-axis,
  \item  $ \breve{\rho}= 0 $ corresponds to the pair of intersecting lines $|x|- |y|= 0 $,
  \item  For any $ 0 < c < \infty $, $ \breve{\rho}=c $ corresponds to the hyperbola $x^2-y^2=c^2 $.
\end{itemize}
Multiplication and division  in hyperbolic form is done as follows:
\begin{equation}\label{eq:multanddivinhypform}
   \breve{\rho_1}e^{\theta_1 E} \breve{\rho_2}e^{\theta_2 E}= \breve{\rho_1}\breve{\rho_2} e^{( \theta_1 +\theta_2 ) E},\quad \frac{\breve{\rho_1}e^{\theta_1 E}}{\breve{\rho_2}e^{\theta_2 E}} = \frac{\breve{\rho_1}}{\breve{\rho_2}}e^{( \theta_1 - \theta_2 ) E}.
\end{equation}

\begin{remark}
 In this paper, we named the number $\breve{\rho}$ as \emph{hyperbolic modulus} of  $ xI+yE \in  \mathcal{H} \subset LC_2 $. This is a specific case of the \emph{hyperbolic norm} defined on real Euclidean space  $\mathbb{R} ^n $. See, \cite{JR}.
\end{remark}

\begin{theorem}
Let $ L = xI + yE \in \mathcal{E} \subset LC_2 $ and $n \in \mathbb{N} $. Then
\begin{equation}\label{eq:demovireform}
  L^n = (xI+yE)^n= \breve{\rho}^n (\cosh (n\theta) I +\sinh (n \theta)  E)
\end{equation}
\end{theorem}

\begin{proof}
  The result follows by first writing $ L$ in hyperbolic form and then using identity (\ref{eq:hyperbolicform}) as
   $$ L^n = (xI+yE)^n =   \left( \breve{\rho} e^{\theta E}\right)^n = \breve{\rho}^n e^{n\theta E}=  \breve{\rho}^n (\cosh (n\theta) I + \sinh (n \theta) E) $$
\end{proof}
The equation (\ref{eq:demovireform}) is the hyperbolic counterpart of De Moivre's identity in complex function theory.
\subsection{ The imaginary unit $i$  and rotation, the shift operator $ E$ and reflection}
In  Theorem  \ref{eq:subgrouptheorem},  we have seen that the operation of the shift operator on each element  $L \in \mathcal{H} $ yields a corresponding element $EL \in \mathcal{V} $ and vice versa. The corresponding  pairs of identity and shift components are $(x,y)$ and $(y,x) $ respectively. This is exactly the reflection of a coordinate plane along the line $y=x$. On the other hand, multiplication of a complex number $z$ by the imaginary unit $i$ has a rotational effect that shifts the argument $z$ by $ \frac{\pi}{2}  $. that is,
$$  \arg (iz)= \arg (z) + \frac{\pi}{2} . $$

\begin{theorem}
  Let $f(L)=f(xI+yE)=u(x,y)I+v(x,y)E $. Then invertibility of the operator $f(L)$ and analyticity condition are equivalent.
\end{theorem}

\subsection{The set of points of singularity as the degeneracy of families of level curves}
In the field of complex numbers $ \mathbb{C} $, the families of curves with  constant modulus are circles. That is,
$$ |z|^2= |x+iy|^2=  x^2+y^2 =c^2, c > 0   $$
 are circles. The degeneracy of this family of concentric circles is the point $z =0$ with zero modulus. For this point, $\frac{1}{z}=  \frac{1}{x+iy}=\frac{x-iy}{x^2+y^2} $ is undefined. We now investigate the case of $LC_2 $. For the case $ \mathcal{H} \subset LC_2 $, for each constant $c > 0 $ the level curve $\breve{\rho}= c $ corresponds to the hyperbola $x^2-y^2=c^2$. The degenerate case of  families of such hyperbolas is the curve $ \breve{\rho}= 0 $, which represents the pair of intersecting lines $|x|-|y|=0 $.  For the set of points $(x,y) $ on this curve, $ \frac{1}{L}=  \frac{1}{xI+yE}= \frac{xI-yE}{x^2-y^2} $ is undefined.  This gives the highlight of the set of points of singularities as the degeneracy of families of curves with points of constant modulus (or hyperbolic modulus).

\begin{figure}[!h]
	\centering
	\begin{minipage}[t]{4cm}
		\centering
		\includegraphics[scale=0.5]{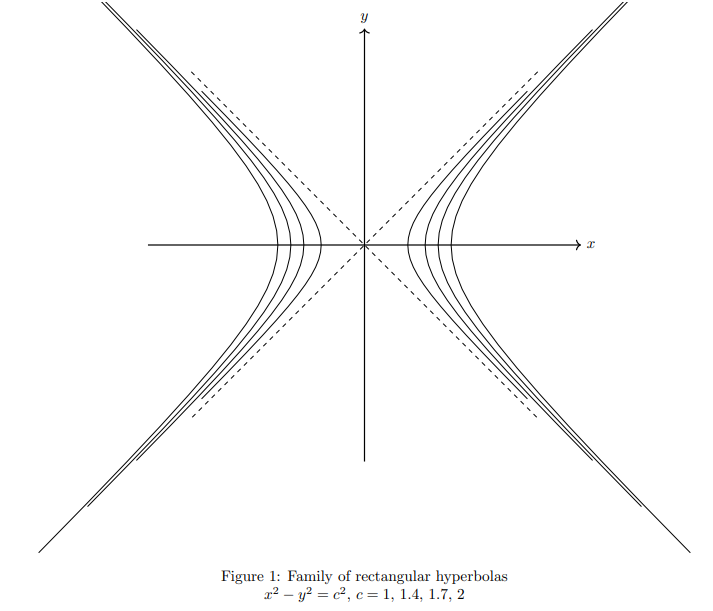}
	\end{minipage}
	\hspace{4cm}
	\begin{minipage}[t]{4cm}
		\centering
		\includegraphics[scale=0.5]{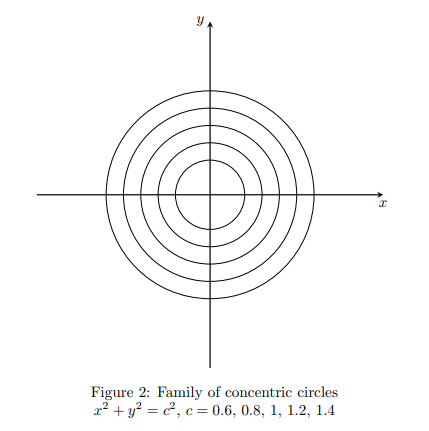}
	\end{minipage}
\end{figure}

\newpage

\subsection{Limits and continuity}
\begin{definition}
Suppose that an $LC_2$ function $f$ is defined in some deleted delta neighborhood of a point $L_0$. An operator  $A \in LC_2 $ is said to be a limit of $f(L)$ at $L_0$ written as $ \lim_{L\rightarrow L_0}f(L)=A $, if for every $\epsilon > 0 $ there exists a $\delta > 0 $ such that $|f(L)-A|< \epsilon $ whenever $0<|L-L_0|< \delta $.
\end{definition}

\begin{theorem}
Let $f: LC_2 \rightarrow LC_2 $  be an $LC_2$ function and $L_0=x_0I+y_0E  \in LC_2 $. If $ f(L)=f(xI+yE)=u(x,y)I+v(x,y)E $, then $  \lim _{L \rightarrow L_0}f(L) = aI+bE $, if and only if
 $$  \lim _{(x,y) \rightarrow (x_0,y_0)}u(x,y) = a,\quad  \lim _{(x,y) \rightarrow (x_0,y_0)}v(x,y) = b,$$
In the usual definition of limits of the functions $u$ and $v$ of two variables.
\end{theorem}

\begin{proof}
The proof is similar to that of a complex function of one variable. Refer, for example, \cite{PS}.
\end{proof}

\begin{definition}
 An operator valued function  $f: LC_2 \rightarrow LC_2 $  is said to be continuous at a point $L_0 \in LC_2 $ if
 $$  \lim _{L \rightarrow L_0}f(L) = f(L_0) .$$
\end{definition}

\subsection{ Differentiability  in the sense similar to that of complex functions does not apply to $LC_2$ functions}
Let $f(L)=u(x,y)I+v(x,y)E $ be a $LC_2$ function. For differentiability $f$  at $L_0=x_0I+y_0E $ in a sense similar to that of a complex function, we evaluate the existence of the limit:
\begin{equation}\label{eq:differentiablityatL0}
   \lim \limits_{ L\rightarrow L_0} \frac{f(L)-f(L_0)}{ L-L_0}
\end{equation}
 along any direction (or path) leading to  the point $(x_0,y_0)$. However, such a situation  will not  always be realistic for every  path. For every $ \delta >  0 $ neighborhood of the point $(x_0,y_0)$ there are an uncountable number of points, namely the set of points $(x,y)$ such that
\begin{equation}\label{eq:badpoints}
   0 < \sqrt{(x-x_0)^2+(y-y_0)^2}< \delta,  \quad   \text{ and } \quad    \frac{y-y_0}{x-x_0}= \pm 1,
\end{equation}
where the expression (\ref{eq:differentiablityatL0}) is undefined. For each point $L_0=x_0I+y_0E$ the $\Delta L = (x-x_0)I+(y-y_0)E $ is not invertible at each point $(x,y)$ on the curve $(x-x_0)^2-(y-y_o)^2=0 $. For this reason, we agree to define differentiability at a point $L_0$ to mean the existence of the limit (\ref{eq:differentiablityatL0}) along any direction except the curve $|y-y_0|=|x-x_0|$ or any curve intersecting this curve. That is, any curve in the open set $LC_2^*$ leading to the point $(x_0, y_0)$.
\begin{definition}
Let $f: LC_2 \rightarrow LC_2 $ be an $LC_2$ function defined in some neighborhood of $L_0=x_0I+y_0E $. We say that $f$ is differentiable at a point $L_0$ if the limit (\ref{eq:differentiablityatL0}) exists along any path that doesn't intersect the curve $|y-y_0|-|x-x_0|=0 $ and leads to the point $L_0$. We denote the derivative of $f$ at $L_0 $  by $f'(L_0)$.
  \end{definition}

\begin{definition}
  The points in the domain of a function $f: D \subset LC_2 \rightarrow LC_2 $ where $f$ is not differentiable are called singular points.
\end{definition}

\begin{example}
  The function $f(L)=L$ is differentiable at every point $ L_0 $  in the sense of $LC_2$  and that $f'(L_0)= I $.
  In fact, $$ \frac{L-L_0}{L-L_0}= (L-L_0)(L-L_0)^{-1}=I, \quad  \forall (L-L_0) \in LC_2^*. $$
  \end{example}
\begin{example}
  The $LC_2$ function $f(L)= L \overline{L}$ is differentiable at $L=0 $ in the sense of $LC_2$ but not analytic at $L=0 $. Indeed, $$\lim_{L\rightarrow 0} \frac{L \overline{L}}{L}=   \lim_{L\rightarrow 0} \overline{L}=0,\quad \forall L \in LC_2^*. $$
  \end{example}

\begin{example}
  The function $f(L)=\sin L $ is differentiable in the sense of $LC_2$ at every point $ L_0 $  and that $f'(L_0)=  \cos L_0 $.
  \begin{equation}\label{eq:oddpowerdifference}
      L^{2n+1}- L_0^{2n+1} = (L -L_0) \sum_{k=0}^{2n} L^{2n-k}L_0^{k}
  \end{equation}
 and by (\ref{eq:sineandcosine}) and (\ref{eq:oddpowerdifference})
 \begin{equation}\label{eq:sinediffquotient}
  \frac{ \sin L - \sin L_0 }{L-L_0} = (L -L_0)\sum_{n=0}^{\infty} \frac{ \sum_{k=0}^{2n}(-1)^n L^{2n-k}L_0^{k}}{(2n)!(L-L_0)}
 \end{equation}
   Now applying $\lim  L \rightarrow L_0$ to (\ref{eq:sinediffquotient}) in $LC_2^* $ and by  (\ref{eq:sineandcosine}), we get the desired result.
 \end{example}

\begin{example}
     We know that the complex function  $g(z)= x-iy$ is not differentiable in the sense of $ \mathbb{C} $ functions. Let $ f(L)= \bar{L}=xI-yE $.  Then we show that $f$ is not differentiable in the sense of $LC_2 $.
    \begin{equation}\label{eq:conjugatedifferencequotient}
      \frac{\overline{L}-   \overline{L_0}}{L-L_0}= \frac{(x-x_0)I-(y-y_0)E}{(x-x_0)I+(y-y_0)E}.
    \end{equation}
    Now, by putting $x=x_0$ and letting $\lim  y \rightarrow y_0$ we get the limit value $-I$. On the other hand, by putting $y=y_0$ and letting $\lim  x \rightarrow x_0$ we get the limit value $I$. Both of these two paths lie inside $LC_2^*$. This shows that $f(L)= \overline{L}$ is not differentiable at any point $L_0$ in the sense of $LC_2$.
 \end{example}

\begin{definition}
An $LC_2$ function $f$ is said to be differentiable at an arbitrary point $L$ in its domain if the limit
  \begin{equation}\label{eq:differentiablityatL}
   \lim \limits_{ \triangle L \rightarrow 0} \frac{f(L+\triangle L)-f(L)}{\triangle L}
\end{equation}
exists. We denote by $f'(L)$ the derivative of $f(L)$.
\end{definition}

 \begin{definition}
 An $LC_2$ function $f$ is said to be \emph{holomorphic in the sense of $LC_2$} in an open set $ D \subset LC_2$ if it is differentiable at each point $L \in D $ in the sense of $LC_2$. The function $f$ is said to be holomorphic at a point $L_0 $ if there exists some open neighborhood  $D$ of $L_0$ on which $f$ is holomorphic in $D$ in the sense of $LC_2 $.
 \end{definition}
The Cauchy-Reimann equations are necessary conditions for the analyticity of complex functions. See, for example, \cite{DP}, \cite{SL}. We set a counterpart of the $LC_2$ functions associated with a \emph{Cauchy-Reimann type equations}.
\begin{theorem}
Suppose that $f(L)=f(xI+yE)=u(x,y) I+v(x,y)E$ and that $\frac{df}{dL}$ exists at the point $ L_0= x_0I+y_0E $.
The necessary condition for the operator function $f(L)$ to be analytic in the region $\Omega \subset \mathbb{ R}^2 $ is that the first-order  partial derivative of $u$ and $y$ exist and satisfy the Cauchy-Riemann type equations:
\begin{equation}\label{eq:CRtypr}
  u_x=v_y,\quad u_y = v_x .
\end{equation}
\end{theorem}
\begin{proof}
  Apply the limits $\triangle L \rightarrow 0 $  for two specific cases, one with $\triangle x =0 $ and  $\triangle y \rightarrow 0 $ and the second with $ \triangle y =0 $ and  $\triangle x \rightarrow 0 $ . The result follows by equating the two results.
\end{proof}

\subsection{Comparative analysis of analyticity in the sense of $\mathbb{C}$ and in the sense of $LC_2$}
Suppose that $f(L)=f(xI+yE) = u(x,y)I+v(x,y)E $  is defend as an operator valued function on some subset of $LC_2 $, and that $g(z)=f(x+iy) = u(x,y)+ iv(x,y) $   is defend as a complex-valued function on some subset of $\mathbb{C} $. The identity component of the function $f(L)$ and the real part of  $g(z)$ are identical and are equal to $u(x,y)$; the shift component of the function $f(L)$ and the imaginary part of $g(z)$ are identical and are equal to $v(x,y)$. Note that $f$ and $g$  may not be the same. We see some practical examples where both $f$ and $g$ are analytic, either of them is analytic or none of them is analytic.

 \begin{example}
    Let $ f(L)= L^2=(xI+yE)^2= (x^2+y^2)I+(2xy)E $, and $g(z)= (x^2+y^2)+(2xy)i $. Then $f$ is  analytic in the sense of $LC_2 $, $g$ is not analytic in the sense of $ \mathbb{C} $.
 \end{example}

 \begin{example}
    Let $ f(L)= (x^2-y^2)I+(2xy)E $, and $g(z)= (x^2-y^2)+(2xy)i $. Then $f$ is not analytic in the sense of $LC_2 $, $g$ is analytic in the sense of $ \mathbb{C} $.
 \end{example}
If a function $f(z)=u(x,y)+iv(x,y)$ is analytic in a domain $D$ then its real part $u$ and imaginary part $v$ are harmonic. That is, they satisfy Laplace's equation. See, for example, \cite{DP}.

\begin{definition}[\textbf{Wave equation}]
 Let $u=u(x,y) $. Then the patrial differential equation
  $$ \frac{\partial^2 u}{\partial x^2}  -\frac{\partial^2u}{\partial y^2}:= \Box u= 0, $$
 where, $ \Box :=   \frac{\partial^2 }{\partial x^2} - \frac{\partial^2 }{\partial y^2} $ is the wave operator (in two-dimensional space), which is called the wave equation.
\end{definition}

In complex functions theory, if  a function $ f(z)$ of a complex argument $z $  is analytic then $\frac{\partial f }{\partial \bar{z}} =0 $. See, for example, \cite{SL}. In the next theorem, we show the following counterpart for the operator-valued functions.
\begin{equation}\label{eq:partialfbypartialL}
  \frac{\partial f}{\partial L}  =\frac{\partial f}{\partial x} \frac{\partial x}{\partial L}+  \frac{\partial f}{\partial y}  \frac{\partial y}{\partial L} = \frac{1}{2}\left(I \frac{\partial }{\partial x} +  E \frac{\partial }{\partial y}  \right)f
\end{equation}

\begin{equation}\label{eq:partialfbypartialLbar}
  \frac{\partial f}{\partial \bar{L} }  =\frac{\partial f}{\partial x} \frac{\partial x}{\partial \bar{L} }+  \frac{\partial f}{\partial y}  \frac{\partial y}{\partial \bar{L} } = \frac{1}{2}\left( I \frac{\partial }{\partial x} -  E\frac{\partial }{\partial y}  \right)f
\end{equation}

\begin{align}\label{eq:partialLpartialLbar}
  \frac{\partial }{\partial L}  \frac{\partial }{\partial \bar{L} }&=
    \frac{1}{4}\left(I \frac{\partial }{\partial x} +  E \frac{\partial }{\partial y}\right) \left( I \frac{\partial }{\partial x} -  E\frac{\partial }{\partial y}\right) \nonumber  \\
   & = \frac{1}{4}\left( \frac{\partial^2 }{\partial x^2} - \frac{\partial^2 }{\partial y^2} \right)I  \nonumber  \\
   &= \frac{1}{4}\Box I
\end{align}

\begin{theorem}
  Let $f: LC_2 \rightarrow LC_2$ be an $LC_2$-valued function. The condition
  $$\frac{\partial f }{\partial \bar{L} } =0,  $$
  is a necessary for  $f$ to be  analytic in the $LC_2$ sense.
\end{theorem}
 \begin{proof}
   Let $f(L)=f(xI+yE)= u(x,y)I+ v(x,y)E $.  By (\ref{eq:partialfbypartialLbar}) and  by Cauchy-Reimann type equations (\ref{eq:CRtypr}),
   \begin{align*}
    \frac{\partial f}{\partial \bar{L} } & = \frac{1}{2}(u_x I  + v_x E )I -  \frac{1}{2}  (u_y I  +  v_y E )E \\
      & = \frac{1}{2}(u_x   - v_y )I +  \frac{1}{2} (v_x-u_y) E \\
      & = 0.
   \end{align*}
   \end{proof}
\begin{theorem}
Let $f(L)$ and $g(L)$ have identity and shift components with continuous first-order partial derivatives that satisfy the Cauchy-Riemann type condition. Then  the sum, products and composition have the identity, and the shift components have continuous first-order partial derivatives that satisfy the Cauchy-Riemann condition.
\end{theorem}
\begin{proof}
   Let $f(L)=u_1 I + v_1E$, and $ g(L)=u_2I+v_2 E $ have identity and shift components whose first order partial derivatives are continuous and satisfy the Cauchy-Riemann type equations. We have the product:
   \begin{equation}\label{eq:productfandg}
      f(L)g(L)= (u_1 I + v_1E)(u_2I+v_2 E)= u I+v E,
   \end{equation}
 where
  \begin{equation}\label{eq:productcomponent}
     u= u_1u_2 + v_1v_2,\quad  v= u_1v_2 + u_2v_1,
  \end{equation}
from which (\ref{eq:productcomponent}) by using the cauchy-Riemann type equations given in (\ref{eq:CRtypr}), we get
\begin{equation*}
\begin{aligned}[c]
u_x&=u_{1x}u_2+u_1u_{2x}+v_{1x}v_2+v_1v_{2x}\\
&=v_{1y}u_2+u_1v_{2y}+u_{1y}v_2+v_1u_{2y}\\
&=(v_1u_2+u_1v_2)_y\\
&=v_y,
\end{aligned}
\qquad
\begin{aligned}[c]
u_y &=u_{1y}u_2+u_1u_{2y}+v_{1y}v_2+v_1v_{2y}\\
&=v_{1x}u_2+u_1v_{2x}+u_{1x}v_2+v_1u_{2x}\\
&=(v_1u_2+u_1v_2)_x\\
&=v_x.
\end{aligned}
\end{equation*}
For the composition of $f$ and $g$
  \begin{align}\label{eq:composition}
      f(g(L))&=f(u_2I+ v_2E )  \nonumber \\
     & = u_1(u_2,v_2)I+v_1(u_2,v_2)E    \nonumber  \\
     & = uI+vE,
  \end{align}
  where $ u= u_1(u_2,v_2) $ and $ v = v_1(u_2,v_2)$.
\end{proof}

\begin{align}\label{eq:partialupartialx}
  \frac{\partial u}{\partial x} & = \frac{\partial u_1}{\partial u_2}  \frac{\partial u_2}{\partial x} + \frac{\partial u_1}{\partial v_2}  \frac{\partial v_2}{\partial x}  \nonumber \\
   & = \frac{\partial v_1}{\partial v_2}  \frac{\partial v_2}{\partial y} + \frac{\partial v_1}{\partial u_2}  \frac{\partial u_2}{\partial y}   \nonumber    \\
   & = \frac{\partial v }{\partial y},
\end{align}
\begin{align}\label{eq:partialupartialy}
  \frac{\partial u}{\partial y} & = \frac{\partial u_1}{\partial u_2}  \frac{\partial u_2}{\partial y} + \frac{\partial u_1}{\partial v_2}  \frac{\partial v_2}{\partial y}  \nonumber \\
   & = \frac{\partial v_1}{\partial v_2}  \frac{\partial v_2}{\partial x} + \frac{\partial v_1}{\partial u_2}  \frac{\partial u_2}{\partial x}   \nonumber    \\
   & = \frac{\partial v }{\partial x}.
\end{align}
In addition, all first order partial derivatives of $u $ and $v$ are continuous. Hence the Cauchy-Reimann type equations are continuous and the first order partial derivative of $u$ and $v$ are continuous.

\subsection{Line integrals of $LC_2 $  functions }
\begin{definition}
Let $C$ be a smooth  curve in  the $LC_2$ plane that may be parameterized as
 $$ C= (x(t),y(t)),\, t_0 \leq t \leq t_1 .$$
Then we have

\begin{align*}
  &\int_{C} f(L)dL = \int_{t_0}^{t_1}[u(x(t),y(t))I+v(x(t),y(t))E][x'(t)I+y'(t)E]dt \\
   & = I \int_{t_0}^{t_1}[u(x(t),y(t))x'(t)+ v(x(t),y(t))y'(t)]dt + E \int_{t_0}^{t_1} [v(x(t),y(t))x'(t)+ u(x(t),y(t))y'(t)]dt.
\end{align*}
\end{definition}

\begin{theorem}
  Suppose that $f(L)=f(xI+yE)=u(x,y)I+v(x,y)E$ is an operator valued function such that $u,v \in C(\Omega)$  for some domain $ \Omega \subset \mathbb{R} ^2 $. Let $C$ be a piecewise smooth, simply connected, closed curve lying inside $C$. Then
  $$ \oint  _{C}f(L)dL =0 $$
\end{theorem}
\begin{proof}
By using the Cauchy- Riemann conditions satisfied by the analytic function $f$ and Green's theorem, we have
  \begin{align*}
     \oint_{C}f(L)dL  & = \oint_{C}(uI+vE)(dxI+dyE) \\
     & =  I\oint_{C}(udx+vdy)+ E \oint_{C}(udy+vdx) \\
     & = I \iint_{R}(v_x-u_y)dA+ E \iint_{R}(u_x-v_y)dA=\textbf{0}.
  \end{align*}
\end{proof}

\begin{example}
  Let $C$ be a closed curve  which forms the boundary of a triangular region with vertices $(0,0), (1,0), (0,1) $. Show that
  $$ \oint_{C} L dL=0 .$$
  We have $$ \oint_{C} L dL= I\int_{0}^{1}x dx + \int_{1}^{0} (x I+ (1-x) E)dx+ E \int_{1}^{0}ydy = \textbf{0}. $$
\end{example}

\begin{example}
Let $C_R $ a circle of radius $R$, centered at $(0,0)$. Then
  \begin{equation*}
    \oint_{C_R} \bar{L}dL= 2\pi R^2 E.
  \end{equation*}
  $$ \oint_{C_R} \bar{L}dL= \oint_{C_R} (xI-yE)(dxI+dyE)=I \oint_{C_R} xdx-ydy + E \oint_{C_R}xdy-ydx .$$
  The required result can obtained by applying Green's theorem to the two closed line integrals.
\end{example}

\subsection{A version of Cauchy's integral formula for $ \mathbb{C} $  does not work for $LC_2 $ }
In this subsection, we show that a version of the well known Cauchy integral formula for complex functions  is not applicable for the case of an operator valued function defined on subset of $LC_2 $.
   \begin{theorem}[Cauchy's integral formula]
Suppose that $f$  is analytic in a simply-connected domain $ \Omega $ and $C $ is any simple closed contour lying entirely within $ \Omega $. Then for any point $ z_0 $ within $C$,
 $$  f(z_0)= \frac{1}{2 \pi i} \oint_{C} \frac{f(z)dz}{z-z_0}. $$
   \end{theorem}
However, a similar version will not apply in the case of $ LC_2 $.
\begin{theorem}\label{eq:CIF}
Let $C$ be a smooth closed curve such that for each  point $L$ on $C$, $L-L_0 \in LC_2 ^*$, where $L_0 = x_0 I +y_0E $, and $ L = xI+yE $ is identified by a point $(x,y)$ on $C$. Then
  $$ \oint_C \frac{dL}{L-L_0} = \textbf{0}, $$
 \end{theorem}
\begin{proof}
  \begin{align*}
    \oint_C \frac{dL}{L-L_0} & =   \oint_C \frac{(\bar{L}-\bar{L_0})dL}{(x-x_0)^2-(y-y_0)^2}\\
     & =   \oint_C \frac{(x-x_0)I-(y-y_0)E (dxI+dy E}{(x-x_0)^2-(y-y_0)^2}\\
     & = I \oint_C \frac{(x-x_0)dx-(y-y_0)dy }{(x-x_0)^2-(y-y_0)^2} + E \oint_C \frac{(x-x_0)dy-(y-y_0)dx }{(x-x_0)^2-(y-y_0)^2}= \textbf{0}
  \end{align*}
 Then, by applying Green's theorem to each of the two closed line integrals, we can observe that the two closed line integrals yield zero.
\end{proof}
 More generally, we have the following result.
\begin{theorem}
Let $f$ is holomorphic function in the sense of $LC_2 $ in some open region enclosing the closed curve $C$,
$$ \oint_C \frac{f(L)dL}{L-L_0}=\begin{cases}
                                 0, & \mbox{if } L-L_0 \in LC_2^*, \forall L \in C,  \\
                                  & \mbox{undefined otherwise}.
                               \end{cases}  $$
\end{theorem}

\begin{remark}
  If the point $(x_0,y_0)$ lies in the interior of the closed curve $C$ or if the pair of perpendicular lines  $ x-y= x_0-y_0 $ and $ x+y= x_0+y_0$  have a common point with $C$, then the closed integral $ \oint_C \frac{dL}{L-L_0} = \textbf{0}, $ diverges. Therefore, the version of the Cauchy integral formula for complex functions is not applicable for $LC_2 $.
\end{remark}

\begin{table}[h!]
    \begin{center}
  \caption{A table of comparison between $\mathbb{C}$ and  $LC_2$.}
        \begin{tabular}{|c|c|c|}
        \hline
        Field or ring & $ \mathbb{C }$ & $LC_2$  \\
            \hline
            Forms of elements & $z=x+iy$  &  $$ L= xI+yE \\
            \hline
            Identity element   & $1$ & $I$ \\
            \hline
            The imaginary unit and the unit shift & $i$  & $E$ \\
            \hline
             Rlations on $i$ and $1$, $I$ and $E$ & $i^2= -1$  &  $E^2=I$\\
            \hline
             Connection with some periodic space  & $\mathbb{AP}_2$  &  $ \mathbb{P}_2$ \\
            \hline
             Function form &  $f(z)=u(x,y)+iv(x,y)$ &  $f(L)=u(x,y) I +v(x,y) E $  \\
            \hline
            Cauchy- Reimann conditions & $u_x=v_y,\, u_y=-v_x $  &   $ u_x=v_y,\, u_y = v_x $\\
            \hline
             Polar form & $z= \rho e^{i\theta}  $  &  $   \breve{\rho} e^{E \theta} $  \\
            \hline
             Modulus and hyperbolic modulus & $ \rho = \sqrt{x^2+y^2}$  & $\breve{\rho} = \sqrt{x^2-y^2}  $    \\
            \hline
             Argument and hyperbolic argument & $\theta= \arctan (y/x)$   &  $ \theta= \atanh (y/x)$ \\
             \hline
         Euler identity and Euler type identity & $ e^{i\theta}= \cos \theta + i \sin \theta $   &  $  e^{E\theta}= \cosh \theta I + E \sinh \theta    $ \\
              \hline
             Necessary condition for Analycity  & $ \frac{\partial f}{\partial \bar{z}}= 0 $  & $ \frac{\partial f}{\partial \bar{L}}= 0 $     \\
            \hline
             The $u$ and $ v$  of analytic $f$ satisfy & $u_{xx}+ u_{yy}=0,\,v_{xx}+ y_{yy}=0  $  &  $ u_{xx}- u_{yy}=0,\,v_{xx}- y_{yy}=0  $  \\
            \hline
             level curves of coordinates  & $\rho= c,\, c>0 $ are circles  &  $ \breve{\rho}= c,\, c>0 $ are hyperbolas  \\
            \hline
             Transformational effect of  $i$ and $E$  & $iz$  is a rotation of $\pi/2$ &  $EL$ is relflection along $y=x $\\
            \hline
             Field or ring properties & $\mathbb{C}$  is a field.  & $LC_2 $ is a ring that is not a field.\\
            \hline
             Isomorphism with spaces of matrices & $ \mathbb{C} \cong \mathbb{AS}_2 $ & $ LC_2  \cong CM_2 $ \\
            \hline

             \end{tabular}
    \end{center}
\end{table}
\newpage
\section{Discussion of the results}
In this paper, we have established a comparative analysis of the space $LC_2$ all operators of the form $ xI+yE $ that are linear combinations of the identity operator $I$ and  a shift operator $E$ acting on the space of all 2-periodic functions $\mathbb{P}_2$. We can identify  an operator $xI+yE$ by a pair of real numbers $(x,y) \in \mathbb{R}^2 $. So, $LC_2 $ can be considered a two-dimensional space generated by $I$ and $E$ . We have
$$ E^{2n-2}=I,\quad  E^{2n-1} = E,\, n \in \mathbb{N}, $$
so that all higher powers of the shift operator $E$ as well, fall into the space $LC_2$. We have shown some analogous structures between $LC_2$ and the complex field $\mathbb{C} $. We have also shown their remarkable differences. We have summarized the comparative analysis in tabulated form. We have introduced the analyticity of $LC_2$ functions in the sense $LC_2$. This gives rise to the Cauchy-Reimann-type equations that are well known in complex analysis. As the real and imaginary parts of a complex analytic function are harmonic, the shift and the identity parts of an analytic $LC_2$ function satisfy the wave equation. In a similar manner to that of $LC_2 $, we have defined the complex field  $\mathbb{C}$  as the set operators of the form $ xI+yE $ of a linear combination of the identity operator $I$ and a shift operator $E$ acting on the space of 2-antiperiodic functions $\mathbb{AP}_2$.
$$ E^{2n-2}= (-1)^{n-1}I,\quad  E^{2n-1} = (-1)^{n-1} E, \, n \in \mathbb{N}, $$
so that all higher powers of the shift operators as well, fall into the space $ \mathbb{C} $.

In the theory of second order partial differential equations, we have elliptic, parabolic, and hyperbolic partial differential equations. The elliptic partial differential equations have no characteristic curves, the parabolic partial differential equations have one family of characteristic curves, and the hyperbolic partial differential equations have two families of characteristic  curves. See, for example, \cite{HW}. The characteristic families of  curves can be contrasted with the set of points of singularities of the spaces that we have studied in this paper. The field of complex numbers $\mathbb{C}$ has only one  singular point, namely $z=0$, has no singularity curve. It may be contrasted with the elliptic partial differential equations. In addition to that, the real and imaginary parts of a holomorphic functions are harmonic. That is, they satisfy  Laplace's equation, which is an elliptic partial differential equation. The space $LC_2$ has singularises along the pair of intersecting lines $x^2-y^2=0 $ and can be contrasted with the  hyperbolic partial differential equations that posses two families characteristic curves. In addition to that, the identity component and the shift component of  $LC_2$ functions are holomorphic in the sense of $LC_2$  satisfy wave equations, which are hyperbolic partial differential equations. Left out of this comparison are the parabolic partial differential equations, which have one characteristic family of curves. Let us define a set of operators
$$  \mathbb{PA} := \{xI+ yE : E^2=0 \},$$
on the space $\{f :\mathbb{R}\rightarrow \mathbb{R} \}$ of all real valued functions, not necessarily periodic. In a similar manner we may define the set of matrices of the form
\begin{equation}\label{eq:characterstic2}
    \mathbb{A}_2:= \left\{ \begin{pmatrix}
   x & y   \\
   0 & x
 \end{pmatrix},\, x,y \in \mathbb{R} \right\} .
   \end{equation}
Then we have
$$ \mathbb{PA} \cong  \mathbb{A}_2. $$
The set of points of singularities of $\mathbb{PA}$ are the same as the set of all points in $\mathbb{R}^2$ with $x=0$. Therefore, the curve of points of singularity is the $y$-axis. The eigenvalues, determinants, etc. of the operators in $\mathbb{PA}$ may be calculated in a similar way as those of $LC_2 $.

\section*{Conclusions and some comments }
We have defined the space $LC_2$ as the set of operators operating on the space $\mathbb{P}_2 $ of all periodic functions of period 2. The space  $LC_2$  is isomorphic to the ring of  all $2 \times 2 $ real circulant matrices. In the same way, we have introduced a new characterization of the field of complex numbers as the set of operators  operating on the space all antiperiodic functions of antiperiod $2$.  The author believes that  a similar  comparative analysis can be done between the space of operators operating on the space of periodic functions with higher integer periods, and some ring or field structures. The ring of real quaternions, $\mathbb{QT}$, is
\begin{equation}\label{eq:realquaternions}
   \mathbb{QT} := \{ x + y \textbf{i} + z \textbf{j}+w \textbf{k}, x,y,z,w \in \mathbb{R},\quad  \textbf{i}^2 = \textbf{j}^2= \textbf{j}^2= \textbf{i}\textbf{j}\textbf{k}= -1 \}.
\end{equation}
 See for example, \cite{JJR}, \cite{AB}, may be contrasted with


 \begin{equation}\label{eq:ALCfour}
    ALC _4 := \{aI+bE+cE^2+d E^3,  a, b, c, d \in \mathbb{R}, E^4=-I \}
 \end{equation}
 or the space of all 4-antiperiodic functions  $\mathbb{AP}_4$. It is very important to note the differences in the properties that  the sets of some  classes operators display based on the space of functions that they are acting on. For example, it suffices to note the cases of $LC_2$ and $\mathbb{C} $.  It may be useful to consider various coordinate systems to study the spaces and singular points in the spaces as well as the degeneracy of the set of constant coordinate values, etc. Geometric descriptions and some physical applications may be incorporated.

\section*{Conflict of interests}
The author declares that there is no conflict of interest regarding the publication of this paper.

\section*{Funding}
This research work is not funded by any organization or individual.

\section*{Acknowledgment}
The author is thankful to the anonymous reviewers for their constructive and valuable suggestions.

\section*{Data availability}
There is no external data used in this research work apart from the reference materials cited in this paper.


\begin{thebibliography}{99}
\bibitem{AB} Amy Buchman, A Brief History of Quaternions and the Theory of Holomorphic Functions of Quaternionic Variables \url{https://arxiv.org/pdf/1111.6088.pdf}.

\bibitem{CHR} C. H. Richardson, An Introduction to the calculus of Finite Differences, D. Van Nostrand Company, Inc. 1954.
\bibitem{KW}  Dan Kalman and James E. White, Polynomial equations and circulant matrices, The American Mathematical Monthly, Vol. 108, No. 9 (Nov., 2001), pp. 821-840.
\bibitem{DV} Danilo P. Mandic, Vanessa Su Lee Goh,  Complex Valued Nonlinear Adaptive Models, John Wiley \&  Sons Ltd. 2009.
\bibitem{DP} Dennis G. Zill, Patric D. Shanahan, A first course in Complex Analysis with applications, Jones and Bartlett Publishers, Inc., 2003.

\bibitem{GS} Geori E. Shilov, Elementary Real and Complex analysis, Dover Publication, Inc., New York, 1973.

 \bibitem{GN} Gis\`{e}le Mophou,  Gaston M. N Gu\'{e}r\'{e}nkata, Existence of Antiperiodic Solutions to Semilinear Evolution Equations in Intermidiate Banach Spaces, Springer proceeding in Mathematics and statistics vol. 37.

 \bibitem{Et} H. D. Ebbenghause, H. Hermes, F. Hirzebruch, M. Koecher, K. Mainzer, J. Neukirch, A. Prestel, R. Remmert, Numbers,Graduate Texts in Mathematics, Springer-Verlag, New York Inc. 1991.

\bibitem{HBY} Hailu Bikila, Yadeta, Decomposition of Spaces of Periodic function into spaces of Period and Spaces of Antiperiodic functions,\url{https://arxiv.org/abs/2210.00915}.

\bibitem{HBY1} Hailu Bikila, Yadeta, Direct sum decomposition of spaces of periodic functions and some connections between shift operators, periodicity of solutions of difference equations, circulant matrices, cyclotomic polynomials, and roots of unity,  \url{https://arxiv.org/abs/2304.02517}.

\bibitem{HW} Hans F. Weinberger, A First Course in Partial Differential Equations with Complex Variables and Transform Methods, Dover Publications, Inc., New York, 1965.

\bibitem{JM} Jayanta Mukhopadhayay, Image and video processing in the Compressed Domain, CRC Press 2011.

\bibitem{JJR} Joseph J. Rotman, A first course in abstract algebra, third edition, Pearson Prentice Hall, 2006.

\bibitem{KM} Kenneth S. Miller,  An introduction to Calculus of Finite Differences,  Henry Holt and Company New York, 1960.

\bibitem{MT} L. M. Milne-Thomson, The Calculus of Finite Differences, Macmilan and Co. Ltd. 1933.

\bibitem{LB} Louis Brand, Differential and Difference Equations, John Weley \& Sons Inc. 1966.

\bibitem{JR} John G. Ratcliffe, Foundations of hyperbolic manifolds, Third ed. graduate texts in Mathematics, Springer, 2019.

\bibitem{NJ}  Norman B. Haaser, Joseph A. Sullivan Real Analysis, Dover Publications Inc, New York, 1991.

\bibitem{RR} Reinhold Remmert, Theory of Complex Functions, Springer Science + Business Media, LLC 1991.

\bibitem{PS} S. Ponnusamy, Herb Silverman, Complex Variable with Applications, Birkh\"{a}user Boston, 2006.

\bibitem{SL} Serge Lang, Complex Analysis 4th ed. Graduate texsts in Mathematics, Springer Science + Business Media, LLC 1999.

\bibitem{VK} Viladmir V. Kisil, Geometry of M\"{o}bius  Transformations, Elliptic, Parabolic and hyperbolic actions of $SL_2(R)$, Imperial College Press, 2012.
\bibitem{WP}  WIKIPEDIA, The Free Encycopedia,  \url{https://en.wikipedia.org/wiki/Split-complex_number}.
\end{thebibliography}
\end{document}